\documentclass[11pt, oneside]{article}   	%
\usepackage{geometry}                		%
\usepackage{graphicx}				%

\usepackage{hyperref}
\hypersetup{
    colorlinks,
    citecolor=green,
    filecolor=black,
    linkcolor=blue,
    urlcolor=blue
}

	\usepackage{subfigure}							 
	\usepackage{multirow}
	\usepackage{amsmath,amssymb,mathtools}
\usepackage{amsthm}
\usepackage{cleveref}
\usepackage{enumerate}

\newcommand{\Fc}{\mathcal{F}}

\newcommand{\Ic}{\mathcal{I}}

\newcommand{\Lc}{\mathcal{L}}
\newcommand{\Mc}{\mathcal{M}}

\newcommand{\Pc}{\mathcal{P}}

\newcommand{\Tc}{\mathcal{T}}

\newcommand{\Vc}{\mathcal{V}}

\newcommand{\Xc}{\mathcal{X}}

\newcommand{\Rbb}{\mathbb{R}}

\newcommand{\Nbb}{\mathbb{N}}

\newcommand{\Zbb}{\mathbb{Z}}

\newcommand{\level}{\mathrm{level}}
\newcommand{\depth}{\mathrm{depth}}
\newcommand{\rank}{\mathrm{rank}}

\newcommand{\id}{\mathrm{id}}
\newcommand{\mix}{\mathrm{mix}}
\newtheorem{theorem}{Theorem}[section]

\newtheorem{lemma}[theorem]{Lemma}
\newtheorem{proposition}[theorem]{Proposition}

\theoremstyle{remark}
\newtheorem{example}[theorem]{Example}
\newtheorem{remark}[theorem]{Remark}

\DeclareMathOperator*{\esssup}{ess\,sup}
\DeclareMathOperator*{\argmin}{arg\,min}

\usepackage{tikz,pgfplots}
\usetikzlibrary{plotmarks,trees,arrows,shapes}

\usepackage{xcolor}

\title{Approximation by tree tensor networks in high dimensions: Sobolev and compositional functions\thanks{M.B.\ acknowledges funding by
 Deutsche Forschungsgemeinschaft (DFG, German Research Foundation) -- Projektnummern 233630050; 211504053 -- TRR 146; SFB 1060.}}

\author{M. Bachmayr\thanks{Institut f\"ur Mathematik, Johannes Gutenberg-Universit\"at Mainz, Germany} ,  A. Nouy\thanks{Centrale Nantes, Nantes Université, Laboratoire de Math\'ematiques Jean Leray, CNRS UMR 6629, France} , R. Schneider\thanks{Technische Universit\"at Berlin, Germany}}
\date{\emph{Dedicated to Ronald DeVore on the occasion of his 80$^\text{th}$ birthday}}

\begin{document}
\maketitle

\begin{abstract}
This paper is concerned with convergence estimates for fully discrete tree tensor network approximations of high-dimensional functions from several model classes.
For functions having standard or mixed Sobolev regularity, new estimates generalizing and refining known results are obtained, based on notions of linear widths of multivariate functions.
In the main results of this paper, such techniques are applied to classes of functions with compositional structure, which are known to be particularly suitable for approximation by deep neural networks. 
As shown here, such functions can also be approximated by tree tensor networks without a curse of dimensionality -- however, subject to certain conditions, in particular on the depth of the underlying tree. 
In addition, a constructive encoding of compositional functions in tree tensor networks is given.
\end{abstract}   
   
\section{Introduction}

The performance of standard approximations schemes based on splines or wavelets can be characterized by classical notions of Sobolev or Besov smoothness.
In the approximation of functions on high-dimensional domains, such standard methods are too inefficient, which is related to the fact that the associated smoothness classes are too broad: in order to approximate high-dimensional functions with tractable complexity, one needs to exploit more specific features of these functions.
This motivates the analysis of more narrow model classes of functions and of their interplay with corresponding approximation algorithms.
A classical example are sparse grids, whose performance is characterized by model classes of functions of high-order mixed regularity.

Here, we consider approximation algorithms based on tree tensor networks, which are a particular type of low-rank approximation of high-order tensors with favorable numerical properties. We study the performance of such approximations for two types of model classes.
On the one hand, we consider a class of functions that can be written as compositions of lower-dimensional component functions. These compositional functions may represent complex hierarchical decision systems where one agent takes a decision based on the decisions taken by other agents, or complex simulation systems where the inputs of a system are given by the outputs (or states) of other systems  \cite{amaral2014decomposition,marque2019efficient,sanson2019systems}; see also the discussion in \cite{poggio2017why}. This class of functions has been shown  by Mhaskar and Poggio \cite{mhaskar2016deep} to allow for efficient approximations -- with a weak dimension-dependence {under certain conditions} -- by deep neural networks.
To obtain convergence estimates for tree tensor networks, we develop two techniques based on estimates of linear widths and on a direct constructive encoding of compositions. 
On the other hand, to put these convergence results into context, we also revisit the approximation of functions of (mixed) Sobolev regularity by tree tensor networks. By a similar technique based on linear widths, we extend and refine estimates from \cite{Schneider201456}. The approximation of Sobolev functions by tree tensor networks has recently also been considered in \cite{2019arXiv190304234G}; there, however, semidiscrete approximation rates in terms of tensor ranks are obtained from singular value estimates, without discretization in the tensor modes.

The approximations by tree tensor networks that we consider are associated to \emph{dimension trees}, which are assumed to be fixed in advance. An example of such a tree is shown in Figure \ref{fig:example_tree}; in general, for a tensor of order $d$, the set $D= \{1,\ldots,d\}$ of modes is recursively subdivided up to the singletons $\{1\}$,\ldots, $\{d\}$. The set of all nodes resulting from this subdivision is then denoted by $T$. The most common choice here is a binary tree, where each interior node of the tree has two children. A tree tensor network with $T$-ranks bounded by $r = (r_\alpha)_{\alpha \in T}$ is a multivariate function $v$ that admits for each $\alpha \in T$ a representation $v(x) = \sum_{k=1}^{r_\alpha} v^\alpha_{k}(x_\alpha)v^{\alpha^c}_k(x_{\alpha^c})$ for some functions $v^\alpha_k$ and $v^{\alpha^c}_k$ of complementary groups of variables $x_\alpha$ and $x_{\alpha^c}$, $\alpha^c = D\setminus \alpha $. For functions in a Hilbert tensor space equipped with a canonical inner product, such a representation is related to the singular value decomposition of the $\alpha$-matricization (or $\alpha$-unfolding) of $v$, identified with a bivariate function. The  approximability of a function by tree tensor networks is therefore related to the decay of singular values of its $\alpha$-matricizations for each $\alpha\in T$.

\begin{figure}[h]
$$\begin{tikzpicture}[scale=.7]  
\tikzstyle{level 1}=[sibling distance=50mm]
\tikzstyle{level 2}=[sibling distance=20mm]
\tikzstyle{root}=[circle,draw,thick,ball color=black!50!gray]
\tikzstyle{interior}=[circle,draw,solid,thick,ball color=black!50!gray]
\tikzstyle{leaves}=[circle,draw,solid,thick,ball color=black!50!gray]
\tikzstyle{active}=[circle,draw,solid,thick,ball color=black!50!gray]

\node [root,label=above:{$\{1,2,3,4,5\}$}]  {}
  child {node [interior,label=above left:{$\{1,2,3\}$}]  {}
    child { node [leaves,label=below:{$\{1\}$}] {}} 
    child { node [interior,label=above right:{$\{2,3\}$}] {}
    child { node [leaves,label=below:{$\{2\}$}] {}}
    child { node [leaves,label=below:{$\{3\}$}] {}}   }      }
  child {node [interior,label=above right:{$\{4,5\}$}]  {}
   child { node [leaves,label=below:{$\{4\}$}] {}}
    child { node [leaves,label=below:{$\{5\}$}] {}}  
   }
       ;
\end{tikzpicture}$$
\caption{Example of a  dimension partition tree $T$ over $D=\{1,2,3,4,5\}$.}
\label{fig:example_tree}
\end{figure}
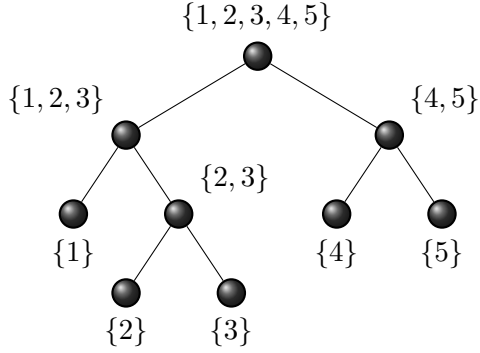

 The results on approximation of certain compositional functions by neural networks in \cite{mhaskar2016deep} are also based on the notion of (binary) dimension trees: the class of approximands considered there is comprised of functions that are compositions with a tree structure. For instance, the tree in Figure \ref{fig:example_tree} corresponds to compositions of the form
\[
  f(x) =  f_D \bigl(f_{\{ 1,2,3\}} \bigl(x_1, f_{\{2,3\}} (x_2,x_3) \bigr), f_{\{4,5\}}(x_4, x_5) \bigr) ,
\]
where the tree being binary corresponds to composing bivariate functions, and where the constituent functions are assumed to be at least Lipschitz continuous.

The general result from \cite{mhaskar2016deep} for approximating such compositions with {an} underlying tree of depth $L$ can be paraphrased as follows: \emph{ Assume that $f$ has compositional structure according to a binary dimension tree with $L$ levels, where each component function is Lipschitz continuous with Lipschitz constant $B > 0$ and has $s$ weak derivatives in $L^\infty$. Then for any smooth, non-polynomial activation function, there exists a neural network $\tilde f$ such that $\| f - \tilde f\|_{L^\infty} \leq \varepsilon$ with $\mathcal{O}( LB^L  \varepsilon^{-2/s} )$ coefficients. }

Note that since $B\leq 1$ is assumed in \cite{mhaskar2016deep}, the dependence on $L$ is not explicitly mentioned there. The dependence of $L$ on $d$ depends on the tree structure, with the most favorable dependence $L \sim \log d$ for a balanced tree: in this case, $B^L$ is polynomial in $d$. The proof is based on the following estimate: for functions $f, g,h$ satisfying the above assumptions with approximations $\tilde f, \tilde g, \tilde h$, one has
\begin{equation}
  \begin{aligned}
      \| f(g, h) - \tilde f( \tilde g , \tilde h) \|_{L^\infty} & \leq  \| f(g, h) - f(\tilde g , \tilde h) \|_{L^\infty} + \| f(\tilde g , \tilde h)  - \tilde f( \tilde g , \tilde h) \|_{L^\infty}  \\
        & \leq B ( \|g - \tilde g \|_{L^\infty} + \| h - \tilde h \|_{L^\infty} ) + \| f - \tilde f \|_{L^\infty}.
  \end{aligned} 
\end{equation}
Applying this estimate recursively starting from the root of the tree, the bound for the approximation complexity follows, using that each component function can be approximated separately by a neural network with $\mathcal{O}(\varepsilon^{-2/s})$ parameters; the composition of these approximations is then again a neural network.

One of the main results of the present work is that a very similar approximation complexity for this class of compositional functions can be achieved by approximations by tree tensor networks, with error measured in $L^2$ (for arbitrary $s$) or $L^\infty$ (with the restriction $s\le 2$). 
More specifically, we show that a tree tensor network approximation $\tilde f$ (that is, a composition of \emph{multilinear} mappings according to the same binary tree structure as the approximand) can be found such that accuracy $\varepsilon$ is achieved with $\mathcal{O}( L^3 B^{3L}  \varepsilon^{-3/s} )$ coefficients, possibly up to terms logarithmic in $\varepsilon$ that depend on the particular construction, and up to a constant polynomial in $d$. In other words, we obtain a very similar dependence on $d$ with $d$-independent convergence rate for tree tensor network approximations, which are substantially easier to handle numerically than approximations by neural networks. In fact, these tensor approximations can be constructed explicitly in certain cases.
The curse of dimensionality is thus shown to be avoided for tree tensor networks under very similar conditions as for deep neural networks.

\medskip
The outline of the paper is as follows. In  \Cref{sec:tree-tensor-networks}, we recall the definition of tree tensor networks and provide upper bounds for the best approximation error of a function in $L^2$ in terms of linear widths. 
In \Cref{sec:sobolev}, using these upper bounds based on linear widths, we provide approximation results  for functions with (mixed) Sobolev regularity. 
Finally in \Cref{sec:composition}, we consider the approximation of compositional functions by tree tensor networks and discuss the conditions under  which the curse of dimensionality is avoided. For the approximation in $L^2$, our proof is based on estimates of linear widths of compositional functions, while for the approximation in $L^\infty$, we use a constructive proof and provide an explicit encoding of an approximation that achieves the announced convergence rates.

\section{Linear widths and tree tensor networks}\label{sec:tree-tensor-networks}

In this section, we first discuss notions of linear widths in the context of multivariate functions.
We then recall the definition of the model class of functions in tree based tensor format (or tree tensor networks), which is interpreted as a particular class of compositional functions. Finally, in the case of square-integrable functions on the unit cube in $d$ dimensions, we deduce upper bounds of the best approximation error in terms of linear widths.

\subsection{Linear widths and singular value decomposition}\label{sec:widths}

We consider functions defined on the unit cube $\Xc = (0,1)^d$ with $d \geq 2$; other sets $\Xc$ with Cartesian product structure could be treated in the same manner in what follows, but we restrict ourselves to this special case for simplicity.
We denote by $D = \{1,\hdots,d\}$ the set of dimensions. 
Throughout this section, we assume $\alpha$ to be a nonempty strict subset of $D$, and we define $\alpha^c = D\setminus \alpha$.
We set  $\Xc_\alpha = (0,1)^{|\alpha|}$, and for $x = (x_1,\hdots,x_d)\in \Xc$, we write $x_\alpha = (x_\nu)_{\nu\in \alpha} \in \Xc_\alpha$.

For closed subspaces $V$ of Banach spaces $Y$, for the error of best approximation of $u \in Y$ by elements of $V$, we introduce the notation 
\[
  E(u,V)_{Y} = \inf_{v\in V} \Vert u-v \Vert_{Y}.
\]
Recall that the classical Kolmogorov $n$-width of a compact subset $K\subset Y$ then reads
\[
   d_n(K)_{Y} =  \inf_{\dim(V) = n}\, \sup_{u\in K} E(u,V)_{Y},
\]
where the infimum is taken over all $n$-dimensional subspaces $V\subset Y$.

In the following summary of basic notions of related linear widths of multivariate functions, we focus on functions in the tensor product Hilbert space 
\[
  X := L^2(\Xc) = L^2(\Xc_1) \otimes \cdots \otimes L^2(\Xc_d),
 \]
 where we abbreviate $X_\alpha := L^2(\Xc_\alpha)$.
We first note that by the canonical isomorphism $\Mc_\alpha \colon L^2(\Xc) \to L^2(\Xc_{\alpha^c} ; X_\alpha)$, any $f \in L^2(\Xc)$ can be isometrically identified with $f^\alpha := \Mc_\alpha f \in L^2(\Xc_{\alpha^c} ; X_\alpha )$ given by $f^\alpha : x_{\alpha^c} \mapsto f(\cdot,x_{\alpha^c})$.
For a given closed subspace $V \subset  X_\alpha$, we define the projection
\[
   \Pc^\alpha_V f := \Mc_\alpha^{-1}\Bigl(  \argmin_{g^\alpha \in  L^2(\Xc_{\alpha^c} ; V) } \| f^\alpha - g^\alpha \|_{L^2(\Xc_{\alpha^c}; X_\alpha )}  \Bigr) ,
\]
which amounts to applying the $L^2$-orthogonal projection onto $V$ to $f^\alpha(x_{\alpha^c})$ for each $x_{\alpha^c}$. For more details on projections on tensor spaces, see also \cite{Nouy2019}.

We now introduce an average linear width associated to $f$ and $\alpha$ as
\begin{equation}\label{average-linear-width}
  \delta^\alpha_n(f) = \inf_{\dim(V) = n} \left(  \int_{\Xc_{\alpha^c}} E\bigl(f^\alpha(x_{\alpha^c}), V\bigr)^2_{X_\alpha} \, dx_{\alpha^c} \right)^{1/2} .
\end{equation}
As we shall now describe, these widths are closely connected to low-rank approximations of $f$. To this end, we define the compact operator
\[
  \mathcal{S}^\alpha_f \colon X_{\alpha^c} \to X_\alpha, \; v \mapsto \int_{\Xc_{\alpha^c}} f^\alpha\, v \,dx_{\alpha^c} \,.
\]
We then define the $\alpha$-rank of $f$ by
\[
   \rank_\alpha(f) := \dim \operatorname{Range} \mathcal{S}^\alpha_f\,,
\]
which in general may be infinite. Note that $\rank_\alpha(g) \leq n$ implies that $g$ can be written in the form
\[
    \sum_{k=1}^n u_k (x_\alpha)\, v_k(x_{\alpha^c})
\]
with functions $u_k \in X_\alpha$, $v_k \in X_{\alpha^c}$ for $k=1,\ldots, n$.

The operator $\mathcal{S}^\alpha_f$ admits a singular value decomposition (see, e.g., \cite[Section 4.4.3]{hackbusch2019tensor}); let $(\sigma^\alpha_k)_{k\geq 1}$ be the non-increasing, non-negative sequence of singular values. Then it is easy to see that for each $n\in\Nbb$,
\[
   \delta^\alpha_n (f) =  \min_{\rank_\alpha(v) \le n} \Vert f - v \Vert_{X} = \Big(\sum_{k>n} (\sigma_k^\alpha)^2\Big)^{1/2}\,;
\]
in other words, $\delta_n^\alpha(f)$ is the error of $L^2$-best approximation of $f$ of $\alpha$-rank $n$.
Moreover, if $U_n \subset X_\alpha$ is a principal subspace of $\mathcal{S}^\alpha_f$ associated to $n$ largest singular values, then 
\[
    \delta^\alpha_n (f) = \| f - \Pc^\alpha_{U_n} f\|_X = \min_{\dim(V)=n} \| f - \Pc^\alpha_V f\|_X ,
\]
that is, such best approximations of $\alpha$-rank at most $n$ can be obtained from the singular value decomposition. As a further consequence, note that
\begin{equation}\label{deltasymm}
  \delta^\alpha_n = \delta^{\alpha^c}_n, \quad n \in \Nbb.
\end{equation}

\subsection{Tree-based tensor formats}

We next introduce some notions that are fundamental to tree-based tensor formats; for further details, we refer to \cite{hackbusch2019tensor,BSU}.
Let $T$ be a  dimension partition tree over $D=\{1,\hdots,d\} $ (see an example on Figure \ref{fig:example_tree}). For any node $\alpha \in T$, we denote by $S(\alpha) \subset T$ the set of sons of $\alpha$, which forms a partition of $\alpha$. $S(\alpha)$ is either empty or {has cardinality} $\#S(\alpha)\ge 2$. 
If $S(\alpha) = \emptyset$, $\alpha$ is called a leaf of $T$. We let $\Lc(T)$ be the set of leaves of $T$ and write $\Ic(T) = T\setminus \Lc(T)$ for the interior nodes of $T$.  

We let $\level(\alpha)$ be the level of $\alpha$ in $T$.  We use the convention $\level(D) = 0$ and for any $\beta \in T\setminus \{D\}$ such that $\beta \in S(\alpha)$, we define $\level(\beta) = \level(\alpha)+1$. Also, we define the depth of $T$ as $\depth(T) = \max \{\level(\alpha) : \alpha\in T\}.$ We set  $T_\ell = \{\alpha\in T : \level(\alpha) = \ell\}$ for $0\le \ell \le \depth(T)$ and $d_\ell = \# T_\ell$.

\begin{example}[Trivial tree]\label{ex:trivial-tree}
The trivial tree $T = \{D,\{1\},\hdots,\{d\}\}$ has a single interior node $D$ and $\depth(T)=1$.
\end{example}
\begin{example}[Linear binary tree]\label{ex:linear-tree}
The linear binary tree \[  T=\{\{1\},\hdots,\{d\},\{1,2\},\hdots,\{1,\hdots,d-1\},D\} \] satisfies $\depth(T)=d-1$ and $T_\ell = \{\{1,\hdots,d-\ell\},\{d-\ell+1\}\}$ for $1\le \ell \le d-1$.
\end{example}
\begin{example}[Balanced binary tree]\label{ex:balanced-tree}
For a balanced binary tree $T$, $\depth(T) = \lceil\log_2(d) \rceil$. For $\ell \le \depth(T)$, we have  $\#T_\ell \le 2^\ell$ and $\#\alpha \le \lceil \frac{d}{2^\ell} \rceil$ for all $\alpha \in T_\ell.$
\end{example}
	
Let $X$ be a tensor product space of  multivariate functions. For a tuple $r=(r_\alpha)_{\alpha \in T}$ (with $r_D = 1$), 
we define a tree-based tensor format in $X$ as 
$$
\Tc_r^T(X) = \bigl \{v \in X : \rank_{\alpha}(v)\le r_\alpha, \alpha \in T  \bigr\}.
$$ 
Tensors satisfying these rank constraints are also known as \emph{hierarchical tensors} \cite{hackbusch2009newscheme} or as \emph{tree tensor network states} in quantum physics \cite{szalay2015tensor}.
Letting $U = U_1\otimes \hdots \otimes U_d$ be a subspace of $X$, where the $U_\nu$ are finite-dimensional subspaces of functions defined on $\Xc_\nu$, we also define 
$$
\Tc_r^T(U) = \{v \in U : \rank_{\alpha}(v)\le r_\alpha\} = \Tc_r^T(X) \cap U.
$$ 
A tuple $r$ is called \emph{admissible} if $\Tc_r^T(X) \neq \emptyset$.

\subsection{Tree based tensor formats as compositional functions and tensor networks}\label{sec:parametrization-tree-networks}
We let $\{\varphi^\nu_{i}\}_{i=1}^{n_\nu}$ denote a basis of $U_\nu$, and introduce the map $\varphi^\nu : \Xc_\nu \to \Rbb^{n_\nu}$ such that $\varphi^\nu(x_\nu) = (\varphi^\nu_{i}(x_\nu))_{i=1}^{n_\nu}$.
A function $f \in \Tc_r^T(U)$ can be parametrized by a set of multilinear functions $\{G^\alpha : \alpha \in T\}$, where $G^\alpha : \bigtimes_{\beta\in S(\alpha)} \Rbb^{r_\beta} \to \Rbb^{r_\alpha} $ for $\alpha \in \Ic(T)$ is multilinear, and $G^\alpha : \Rbb^{n_\alpha} \to \Rbb^{r_\alpha}$ for $\alpha \in \Lc(T)$ is linear. The function $f$  can be written 
$$
f(x) = G^D((z_\alpha)_{\alpha \in S(D)}),
$$
with $z_\alpha = \varphi^\alpha(x_\alpha)$ for a leaf node $\alpha \in \Lc(T)$, and 
$$
z_\alpha = G^{\alpha}((z_\beta)_{\beta \in S(\alpha)})
$$
for an interior node $\alpha\in \Ic(T)$. 

The multilinear functions $G^\alpha$ can be identified with tensors of order $\#S(D)$ for $\alpha=D$, $1+\#S(\alpha)$ for $\alpha \in \Ic(T)\setminus \{D\}$ and $2$ if $\alpha\in \Lc(T).$ This yields the interpretation of 
the tree-based format as a tree tensor network.

\begin{example}
For the tree $T$ of Figure \ref{fig:example_tree}, $f \in \Tc_r^T(U)$ can be written
$$
f(x) = G^D(G^{\{1,2,3\}}(G^{\{1\}}(z_1),G^{\{2,3\}}(G^{\{2\}}(z_2),G^{\{3\}}(z_3))),G^{\{4,5\}}(G^{\{4\}}(z_4),G^{\{5\}}(z_5)))
$$
where $z_\nu = \varphi^{\nu}(x_\nu)$, $1\le \nu\le d.$
\end{example}
\begin{example}[Trivial tree and Tucker format]
For the trivial tree of \Cref{ex:trivial-tree}, $\Tc_r^T(U)$  corresponds to the Tucker format and $f\in \Tc_r^T(U)$ can be written
$$
f(x) = G^D(\varphi^{1}(x_1),\hdots,\varphi^{d}(x_d)).
$$
\end{example}
\begin{example}[Linear tree and tensor train format]
For the linear binary tree of \Cref{ex:linear-tree}, $\Tc_r^T(U)$ corresponds to the tensor train (TT) Tucker format.
\end{example}	

The number of parameters (or representation complexity) of an element  in $\Tc^T_r(U)$ is 
$$
N(T,r,U) = \sum_{\alpha \in \Ic(T)} r_\alpha \prod_{\beta\in S(\alpha)} r_\beta + \sum_{\nu=1}^d  r_\nu n_\nu, 
$$
with $n_\nu = \dim(U_\nu).$
If $ r_\alpha \le R$ for all $\alpha$ and $\dim(U_\nu) \le n$ for all $\nu$, then 
\[
  N(T,r,U) \le R^a + (\#T-1-d)R^{a+1} + d Rn \le R^a +  (d-2)R^{a+1} + dRn,
\]
where $a = \max_{\alpha \in \Ic(T)} \#S(\alpha)$ is the arity of the tree ($a=2$ for a binary tree, and $a=d$ for a trivial tree).

\subsection{Best approximation error and linear widths}

Let $T$ be a fixed dimension tree and 
$r = (r_\alpha)_{\alpha\in T}$ be an admissible rank. For any subspace $U\subset X = L^2(\Xc)$, 
the error of best approximation of a function $f\in X$ by an element of $\Tc^T_r(U)$ is
$$
e_{r,U}^T(f)_X = \inf_{v \in \Tc^T_r(U)} \Vert f - v \Vert_X\,,
$$ 
and the error of best approximation of a function $f\in X$ by an element of $\Tc^T_r(X)$ is 
 $$ e_{r}^T(f)_X := e_{r,X}^T(f)_X .$$

The following result provides an upper bound of the best approximation error with tree tensor networks in terms of linear widths of $f$. The argument is similar to the one for the discrete case given in \cite{grasedyck2010}.
\begin{proposition}\label{prop:best-approx-rU}
Let $f \in X$ and let $r\in \Nbb^{\#T}$ be an admissible rank. Then
\begin{equation}\label{eq:treeerr1}
e_r^T(f)^2_{X} \le  \sum_{\alpha \in T\setminus \{D\}} \bigl( \delta^\alpha_{r}(f) \bigr)^2.
\end{equation}
Furthermore, 
for any finite-dimensional subspace $U = U_1  \otimes \hdots \otimes U_d$, we have
\begin{equation}\label{eq:treeerr2}
e_{r,U}^T(f)^2_{X} \;\le\; \sum_{\nu=1}^d \int_{\Xc_{\nu^c}} E\bigl(f^\nu(x_{\nu^c}),U_\nu\bigr)_{X_{\nu}}^2 \,d x_{\nu^c} 
 + \sum_{\alpha \in A\setminus \{D\}} \bigl( \delta^\alpha_{r} (f) \bigr)^2 \,,
\end{equation}
with $A = \Ic(T)$ if $\dim(U_\nu)=r_\nu$ for all $1\le \nu\le d$, or $A = T$ otherwise. 
\end{proposition}
\begin{proof}
We first show that for any finite-dimensional subspace $U = U_1\otimes \hdots \otimes U_d $,  and any collection of subspaces $V_\alpha \subset X_\alpha$ with $ \dim(V_\alpha) = r_\alpha,$ $\alpha\in T\setminus \{D\}$, with $A$ as in the hypothesis we have 
\begin{equation}
e^T_r(f)_X \le e^T_{r,U}(f)_X \le  \sum_{\nu=1}^d \Vert f - \Pc^{\{\nu\}}_{U_\nu} f \Vert_{X}^2  + \sum_{\alpha \in A\setminus \{D\}} \Vert f - \Pc^\alpha_{V_\alpha} f \Vert_{X}^2.
\label{best-error-bound-Ualpha-bis}
\end{equation}
The result will be proved by constructing a particular approximation $f_r \in \Tc_r^T(U) $ and by providing an upper bound of $\Vert f - f_r \Vert_X^2$.
We define the approximation 
\begin{equation*}
f_r = \Pc_{L+1} \Pc_{L}\hdots \Pc_{1} f, \label{root-to-leaves-projection}
\end{equation*}
where $L = \depth(T)$,  
$ %
\Pc_{\ell} = \prod_{\alpha\in T_\ell} \Pc^\alpha_{V_\alpha}, 
$ %
for $1\le \ell \le L,$ and $ \Pc_{L+1} = \Pc_{U_1}^{\{1\}}  \hdots \Pc^{\{d\}}_{U_d} $. For disjoint subsets  $\alpha$ and $\beta$, the projections $\Pc^\alpha_{V_\alpha}$ and $\Pc^\beta_{V_\beta}$ commute. Therefore, 
the definition of $\Pc_\ell$ does not depend on the order of projections $\Pc^\alpha_{V_\alpha}$, $\alpha\in T_\ell$.

Let us first prove that $f_r\in \Tc_r^T(U) $. We clearly have $f_r \in U$. 
Then we note that 
for any function $g$ and any pair $\alpha,\beta \in T$ such that $\beta \subset \alpha$ or $\beta \subset \alpha^c$, we have $\rank_{\alpha}(\Pc_{V_\beta} g) \le \rank_\alpha(g)$ for any subspace $V_\beta$ in $X_\beta$. Then for $\alpha \in T$ with level $\ell$, since the projections  
$\Pc_{\ell'}$ with $\ell'>\ell$ only involve projections $\Pc^\beta_{V_\beta}$ with $\beta \subset \alpha$ or $\beta \subset \alpha^c$, we have 
$\rank_{\alpha}(f_r) \le \rank_{\alpha}(\Pc_{\ell} \hdots \Pc_1 f) = \rank_\alpha(\Pc^\alpha_{V_\alpha} g) \le r_\alpha$, where 
$g = \prod_{\beta \in T_\ell,\beta\neq \alpha} \Pc^\beta_{V_\beta} \Pc_{\ell-1} \hdots \Pc_1 f .$ This proves that $\rank_\alpha(f_r) \le r_\alpha$ for all $\alpha \in T$, which implies $f_r \in  \Tc^T_r(X)$. We therefore deduce that $f_r \in  \Tc^T_r(X) \cap U =  \Tc^T_r(U)$.

Now let us provide the desired upper bound for $\Vert f - f_r \Vert_X$. 
For clarity, we let $\Vert\cdot\Vert = \Vert \cdot \Vert_{{X}}.$ Using the properties of orthogonal projections, we have
\begin{align*}
\Vert f - f_r \Vert^2& = \Vert f-  \Pc_{L+1}\hdots \Pc_{1} f \Vert^2 =\Vert f - \Pc_{L+1} f \Vert^2   + \Vert \Pc_{L+1}  (f -  \Pc_{L-1}\hdots \Pc_{1} f)  \Vert^2
\\
&\le \Vert f - \Pc_{L+1} f \Vert^2 +  \Vert f - \Pc_{L}\hdots \Pc_{1} f  \Vert^2
\end{align*}
Repeating the above arguments, we obtain
$
\Vert f - f_r \Vert^2 \le  \sum_{1\le \ell \le L+1}\Vert f - \Pc_\ell f \Vert^2. $ 
For $1\le \ell \le L$, we have
$\Vert f - \Pc_\ell f \Vert^2 = \Vert f - \prod_{\alpha\in T_\ell} \Pc^\alpha_{V_\alpha} f \Vert^2 \le \sum_{\alpha\in T_\ell} \Vert f - \Pc^\alpha_{V_\alpha} f\Vert^2,$ which provides the desired bound for the general case.
In the case where $\dim(U_\nu)=r_\nu$ for $1\le \nu\le d$, the result is deduced from the above result by choosing $V_\nu = U_\nu$ for $1\le \nu\le d$ in the definition of $f_r$, and by defining $\Pc_{L+1}=\id$. 

Now \eqref{eq:treeerr1} follows from \eqref{best-error-bound-Ualpha-bis} by taking 
 the infimum over spaces $U_\alpha$ and $V_\alpha$,
 and \eqref{eq:treeerr2} follows from \eqref{best-error-bound-Ualpha-bis} by taking 
 the infimum over spaces $V_\alpha$.
\end{proof}

\section{Approximation of functions in Sobolev spaces}\label{sec:sobolev}
In this section, we consider the approximation of functions in Sobolev spaces on $\Xc = (0,1)^d$ using tree tensor networks:
on the one hand, the standard fractional Sobolev spaces $H^s(\Xc)$ for $s>0$, and on the other hand, the mixed Sobolev spaces $H^s_\mix(\Xc)$, which can be characterized as tensor products $H^s_\mix(\Xc) = H^s(0,1) \otimes \cdots \otimes H^s(0,1)$ with the canonical cross norm. Assuming a dimension tree $T$ for $D$, we again write $\Xc_\alpha = (0,1)^{|\alpha|}$ for $\alpha \in T$ and abbreviate $H^s_\alpha = H^s(\Xc_\alpha)$ and $H^s_{\alpha,\mix} = H^s_\mix(\Xc_\alpha)$.

\subsection{Sobolev spaces}

We first recall a standard result on Kolmogorov widths of Sobolev balls (see, e.g., \cite[Chapter  VII]{pinkus2012n}).
Here and in what follows, we denote by $B_1(X)$ the unit ball of a given normed space $X$.

\begin{theorem}\label{kolmogorov-width-sobolev-balls}
Let $I = (0,1)^m$. Then 
\[
d_n \bigl(B_1( H^{s}(I) )\bigr)_{L^2} \le R  n^{-s/m} ,
\]
where $R>0$ is independent of  $n$. 
 \end{theorem}
 It is well known that there exist approximation tools, such as splines or wavelets, that achieve the optimal rate of convergence given by the Kolmogorov widths \cite{DeVore98nonlinearapproximation}. In other {words}, there exists a sequence of $n$-dimensional spaces $V_n  \subset L^2(I)$ such that for  all $f\in H^{s}(I)$, 
 $$E(  f ,V_n)_{{X_\alpha}} \le  M n^{-s/m} \Vert f \Vert_{H^{s}}$$
where $M \ge R$ is a constant independent of $f$ and $n$.
For $f \in H^{s}(\Xc)$, $s>0$, from this bound we deduce the following estimate on the average linear widths of $f$ defined in \eqref{average-linear-width}.

\begin{proposition}\label{prop:dp-sobolev}
Let $f\in  H^{s}(\Xc)$, $s>0$. For any $\alpha \in T\setminus \{D\}$, we  have 
\[
  \delta^\alpha_n(f) \le C n^{-s/d_\alpha}\Vert f \Vert_{H^{s}} 
\]
with $d_\alpha = \min\{\#\alpha,d-\#\alpha\}$, and $C$ independent of $r$ and $f$, but depending on $d_\alpha$ and $s$.
\end{proposition}

\begin{proof}
Let $f \in {H^{s}(\Xc)}$. Since $f^\alpha(x_{\alpha^c}) \in H^{s}_{\alpha}$ for almost all $x_{\alpha^c}$, for $x_{\alpha^c}$ such that $f^\alpha(x_{\alpha^c}) \neq 0$ we have
 \[
    E(  f^\alpha(x_{\alpha^c}) ,V_n)_{L^2_{\alpha}} = E\bigl( \|   f^\alpha(x_{\alpha^c}) \|^{-1}_{H^{s}_{\alpha}}   f^\alpha(x_{\alpha^c}) ,V_n \bigr)_{X_{\alpha}} \|   f^\alpha(x_{\alpha^c}) \|_{H^{s}_{\alpha}} ,
 \]
 as well as $E(  f^\alpha(x_{\alpha^c}) ,V_n)_{X_{\alpha}}  = 0$ otherwise. Thus
 \[\begin{aligned}
   \delta^\alpha_n (f)  &=  \inf_{\dim(V_n) = n} \left (\int_{\Xc_{\alpha^c}} E( f^\alpha(x_{\alpha^c}) ,V_n)^2_{X_\alpha} dx_{\alpha^c}\right)^{1/2}  \\
   & \leq  \inf_{\dim(V_n) = n}  \esssup_{x_{\alpha^c} } E \bigl( \|   f^\alpha(x_{\alpha^c}) \|^{-1}_{H^{s}_{\alpha}}   f^\alpha(x_{\alpha^c}) ,V_n \bigr)_{X_{\alpha}}  \| f \|_{L^2(\Xc_{\alpha^c} ; H^{s}_{\alpha} )}  \\
    & \leq d_n\bigl (B_1 (H^{s}_\alpha) \bigr)_{X_\alpha} \| f\|_{H^{s}},
\end{aligned}  \]
where we have used $\Vert f^\alpha \Vert_{{L^2(\Xc_{\alpha^c} ; H^{s}_{\alpha} )}}  \le \Vert  f\Vert_{H^{s}}$.
By \Cref{kolmogorov-width-sobolev-balls}, we have
$
   d_n\bigl (B_1 (H^{s}_\alpha) \bigr)_{X_\alpha}  \leq C_{\#\alpha} n^{-s / \#\alpha }
$
with $C_{\#\alpha}$ independent of $f$ and $n$. The statement now follows with \eqref{deltasymm}.
\end{proof}

For each $\nu \in D$, we introduce a sequence of spaces $U_{\nu,n_\nu}$ with dimension $n_\nu$ ({such as} splines or wavelets) such that for all $u\in H^{s}_{\nu},$
 \begin{align}
 E(u , U_{\nu,n_\nu})_{X_{\nu}} \le M  n_\nu^{-s}\Vert u \Vert_{H^{s}_{\nu}},\label{approx-spaces-Unu}
\end{align}
 which implies 
 \begin{align}
\int_{\Xc_{\nu^c}} E\bigl(f^\nu(x_{\nu^c}),U_{\nu,n_\nu} \bigr)_{X_{\nu}}^2\, dx_{\nu^c}  \le M^2 n_\nu^{-2s} \Vert f \Vert_{H^{s}}^2.\label{uniform-approx-spaces-Unu}
 \end{align}  
Then we let $U_n = U_{1,n_1} \otimes \hdots \otimes U_{d,n_d}$. 
Now, we can deduce an approximation result for the approximation of functions in Sobolev spaces using tree tensor networks.

\begin{theorem}\label{approx-Ws2}
Let $f \in H^{s}(\Xc)$ and $0<\varepsilon<1$, and let $N(f,\varepsilon,d)$ be the minimal complexity $N(T,r,U_n)$ such that 
\[
 e^T_{r,U_n}(f)_{X} \le \varepsilon \Vert f \Vert_{H^{s}}.
 \]
For any dimension partition tree $T$, there exists a constant $C$ depending on $d$ such that 
\[
N(f,\varepsilon,d) \le C \varepsilon^{-d/s}. 
\]
\end{theorem}

\begin{proof}
From \Cref{prop:best-approx-rU} and \Cref{prop:dp-sobolev}, we deduce that
if $$r_\alpha \ge \varepsilon^{-d_\alpha/s} (C_{d_\alpha}\sqrt{\#T-1})^{d_\alpha/s}$$  for each interior node $\alpha \in \Ic(T) \setminus  \{D\}$,  and $r_\nu = n_{\nu} \ge \varepsilon^{-1/s}(M \sqrt{\#T-1})^{1/s}$
for all $1\le \nu \le d$, then 
$$e^T_{r,U_n}(f)_{X} \le \varepsilon \Vert f \Vert_{{H^s}}.$$
The minimal values of ranks such that the above conditions hold are such that $r_\alpha := r_\alpha(\varepsilon) \sim \varepsilon^{-d_\alpha/s}$, $\alpha \in T\setminus \{ D\}$, with constants depending on $d$, $M$ and $s$. Then recalling that $N(f,r,U_{n})=\sum_{\nu=1}^d r_\nu^2 + \sum_{\alpha \in  \Ic(T)  } r_\alpha \prod_{\beta\in S(\alpha)} r_\beta$, we have 
$$
N(f,\varepsilon,d) \lesssim d \varepsilon^{-2/s} + \varepsilon^{-( \sum_{\alpha \in S(D)} d_\alpha)/s} +   \sum_{\alpha \in \Ic(T) \setminus \{D\}} \varepsilon^{-(d_\alpha + \sum_{\beta \in S(\alpha)} d_\beta)/s}.
$$
We note that $\sum_{\alpha \in S(D)} d_\alpha \le \sum_{\alpha \in S(D)} \#\alpha = d$. 
Then consider $\alpha \in \Ic(T) \setminus \{D\}$. If $d_\alpha = \#\alpha$, we have $\#\alpha \le d/2$ and $d_\alpha + \sum_{\beta \in S(\alpha)} d_\beta \le \#\alpha +  \sum_{\beta \in S(\alpha)} \#\beta = 2\#\alpha \le d $. Otherwise,  $d_\alpha = \#\alpha^c $, and we have  $d_\alpha + \sum_{\beta \in S(\alpha)} d_\beta \le \#\alpha^c +  \sum_{\beta \in S(\alpha)} \#\beta =  \#\alpha^c + \#\alpha = d$. Then for any tree, we have 
$
N(f,\varepsilon,d) \lesssim d \varepsilon^{-2/s} + (\#T -d) \varepsilon^{-d/s}
$.
\end{proof}

An important observation is that for any dimension partition tree $T$  the complexity 
$N(f,\varepsilon,d)$ scales as 
$\varepsilon^{-d/s}$, the optimal rate deduced  from linear widths of Sobolev balls. For Sobolev spaces, a shallow network associated with a trivial tree with depth one (Tucker format) has a similar performance  as deep tensor networks associated with binary trees.

\subsection{Mixed Sobolev spaces}

We recall a standard result on Kolmogorov widths of balls of  mixed Sobolev spaces (see e.g. \cite{temlyakov1986approximations}).
\begin{theorem}\label{kolmogorov-width-mixed-sobolev-balls}
Let $I = (0,1)^m$. For any $s>0$, there exists $R>0$ such that for all $n \in \mathbb{N}$,
\[
d_n\bigl(B_1( H^{s}_{\mix}(I) ) \bigr)_{L^2} \le R  n^{-s} \log(n)^{s(m-1)}   \,.
\]
 \end{theorem}

The above result yields the following estimate of the average linear widths of $f$. 
\begin{proposition}\label{dp-mixed-sobolev}
For $f\in  H^{s}_{\mix}(\Xc)$ and $\alpha \in T\setminus \{D\}$, we  have 
$$
\delta^\alpha_n (f) \le C n^{-s} \log(n)^{s(d_\alpha -1)} \Vert f \Vert_{H^{s}_{\mix}}
$$
with $d_\alpha = \min\{\#\alpha,d-\#\alpha\}$, and 
 $C$ a constant independent of $f$ and $n$, but depending on $d_\alpha$ and $s$.
\end{proposition}

\begin{proof}
Let $f \in H^{s}_{\mix}(\Xc)$. Using $\Vert f^\alpha \Vert_{L^2 (\Xc_{\alpha^c} ; H^{s}_{\alpha,\mix} )}  \le \Vert  f\Vert_{H^{s}_{\mix}}$ to argue as in the proof of \Cref{prop:dp-sobolev}, we obtain
\[
    \delta^\alpha_n (f) \leq d_n\bigl (B_1 (H^{s}_{\alpha, \mix}) \bigr)_{X_\alpha} \| f\|_{H^{s}_\mix}.
\]
By \Cref{kolmogorov-width-mixed-sobolev-balls}, 
\[
  d_n\bigl (B_1 (H^{s}_{\alpha, \mix}) \bigr)_{X_\alpha} \leq C_{\#\alpha} n^{-s} \log(n)^{s(\#\alpha-1)} 
\]
with $C_{\#\alpha}$ independent of $f$ and $n$.
The statement follows with \eqref{deltasymm}.
\end{proof}

Another bound is obtained in the next proposition by exploiting results on  
 hyperbolic cross approximation \cite{hansen2012best} (see also \cite{dung2018hyperbolic}). 
 Related conversions from hyperbolic cross approximations to tensor formats have also been considered in \cite[\S 7.6]{hackbusch2019tensor} and \cite{Schneider201456}.
 
\begin{proposition}\label{dp-mixed-sobolev-sparse}
For $f\in H^{s}_{\mix}(\Xc)$ and $\alpha \in T\setminus \{D\}$, we  have 
\[
 \delta^\alpha_n (f) \le C_d\, n^{-2s} \log(n)^{2s(d-2)} \Vert f \Vert_{H^{s}_{\mix}}
\]
with $C_d$ independent of $f$ and $n$, but depending on $s$ and {depending} exponentially on $d$.
\end{proposition}

\begin{proof}
We rely on results on  $m$-term  approximation from \cite{hansen2012best}. We consider the tensor product wavelet system $\{\phi_j\}_{j \in \Ic}$ from \cite[Section 3.2]{hansen2012best}, where $\Ic \subset \Nbb^d \times \Zbb^d$ and where for $  (l,k) \in \Ic$, $\phi_{l,k} (x)=\varphi_{l_1,k_1} (x_1)\hdots  \varphi_{l_d,k_d} (x_d)  $ with $\varphi_{l_\nu,k_\nu}$ a one-dimensional wavelet system.  
Consider $f \in {H^{s}_{\mix} }$ with $\Vert f \Vert_{{H^{s}_{\mix}}} =1$, where $H^{s}_{\mix}$ coincides with the Lizorkin-Triebel space $S^s_{2,2} F$ (see definition in \cite[Section 3.1]{hansen2012best}). It admits an expansion 
 $$
 f = \sum_{j \in \Ic} c_j(f) \phi_j, 
 $$
 with a sequence of coefficients $(c_j(f))_{j\in \Ic}$ in the sequence space $s_{2,2}^s f(\Ic)$ defined in \cite[Definition 3.2]{hansen2012best}.
Then consider the multi-index set $$
  \Ic_L = \{(l,k) \in \Ic : \vert l \vert_1 \le L\}, 
 $$
 which is an hyperbolic cross with cardinality $\# \Ic_L \sim L^{d-1} 2^L$ (see \cite[Remark 5.7]{hansen2012best}).
Then from \cite[Proposition 5.6]{hansen2012best} and the fact that $\Vert f \Vert_{H^s_{\mix}} \sim \Vert (c_j(f))_{j\in \Ic} \Vert_{s_{2,2}^s f}$, 
we have that the approximation
 $$
 f_L = \sum_{j \in \Ic_L} c_j \phi_j
  $$
satisfies
  $
 \Vert f - f_L \Vert_{p} \lesssim 2^{-L s} 
 $.

 We let $L_\alpha((l,k)) = \sum_{\nu\in \alpha} {l_\nu}$, $L_{\alpha^c}((l,k)) =\vert l\vert_1 - L_\alpha((l,k)) $, and define the sets of multi-indices 
 $$
  \mathcal{I}^{\le}_{L} = \{j \in \mathcal{I}_L : L_\alpha(j) \le L_{\alpha^c}(j)\} \quad \text{and} \quad   \mathcal{I}^{>}_{L} = \{j \in \mathcal{I}_L : L_\alpha(j) > L_{\alpha^c}(j)\}.
 $$
 We decompose 
 $$
f_L = f_{L}^{\le} + f_{L}^{>}, \quad f_{L}^{\le} = \sum_{j \in \mathcal{I}_{L}^{\le}} c_{j} \phi_{j}, \quad f_{L}^{>} = \sum_{j \in \mathcal{I}_{L}^{>}} c_{j} \phi_{j}.$$
Defining
$
{\mathcal{I}}_{L,\beta}^{S} = \{j_\beta : (j_\beta,j_{\beta^c}) \in \mathcal{I}_{L}^{S} \}
$,
with $S \in \{ \le, > \}$ and $\beta \in \{ \alpha, \alpha^c\}$,
we have 
$$
f_{L}^{S} = \sum_{j_\beta \in  {\mathcal{I}}_{L,\beta}^{S} } \phi_{j_\beta}(x_\beta) \psi_{j_\beta}^{S,\beta}(x_{\beta^c}), \quad \text{with } \psi^{S,\beta}_{j_\beta}(x_{\beta^c}) = \sum_{j_{\beta^c}: (j_\beta,j_{\beta^c}) \in \mathcal{I}_L^S} c_{(j_\beta,j_{\beta^c})} \phi_{j_{\beta^c}}(x_{\beta^c}),
$$
so that $\rank_\beta(f_L^S) \le \#\mathcal{I}_{L,\beta}^S$.
It follows that
$$
\rank_\alpha(f_L) \le \rank_\alpha(f_L^\le) + \rank_\alpha(f_L^>) = \rank_\alpha(f_L^\le) + \rank_{\alpha^c}(f_L^>) \le 
 \#\mathcal{I}_{L,\alpha}^\le +  \#\mathcal{I}_{L,\alpha^c}^>.
$$
 We observe that for all $j \in \mathcal{I}^{\le}_{L}$, $L \ge L_\alpha(j) + L_{\alpha^c}(j) \ge 2 L_\alpha(j)$, and therefore
 $$\#\Ic^\le_{L,\alpha} \le \# \{(l_\alpha,k_\alpha) :  (l_\alpha,l_{\alpha^c},k_\alpha,k_{\alpha^c})  \in\mathcal{I}  , \vert l_\alpha \vert_1 \le  L/2 \} \lesssim 2^{L/2} (L/2)^{\#\alpha-1}.
 $$ 
  Also, for all $j \in \mathcal{I}^{<}_{L}$, $L >  2 L_{\alpha^c}(j)$, and therefore
 $$\#\Ic^>_{L,\alpha^c} \lesssim  2^{L/2} (L/2)^{\#\alpha-1}. $$
 Since $\max\{\#\alpha,\#\alpha^c\} \le d-1$, we finally deduce that 
 $
 \rank_\alpha(f_L) \le C 2^{L/2} (L/2)^{d-2}
 $
 for some constant $C$. 
 This implies that for $n \ge C 2^{L/2} (L/2)^{d-2} := C x \log_2(x)^{d-2} $, 
 $$
\delta^\alpha_n (f) \lesssim 2^{-Ls} = x^{-2s}.
 $$
A solution to $ x\log_2(x)^{a} = t$, for $t\ge 0$, is given by
 $
x= e^{a W_0(t^{1/a} \log(2)/a)}, 
 $
 where $W_0$ is the principal branch of the Lambert function. 
  Then for $a=d-2$ and $n\ge C t$, we have  
  $$
    \delta^\alpha_n (f) \lesssim e^{-2sa W_0(t^{1/a}\log(2)/a)}.
 $$
 Using \cite[Theorem 2.1]{Hoorfar2008}, we have that for all $t\ge e$, 
 $
 W_0(t) \ge \log(t) - \log(\log(t)),$ which implies 
 $e^{W_0(t)} \ge t \log(t)^{-1}$. Therefore, for $n \ge C (ae/\log(2))^{d-2}$,  
 $$
    \delta^\alpha_n (f) \lesssim ((n/C)^{1/a} \log(2) a^{-1} \log((n/C)^{1/a} \log(2)a^{-1})^{-1}  )^{-2sa} \le  C_{a,s} n^{-2s} \log(n)^{2sa},
 $$
 with a constant $C_{a,s}$ depending on $a$ and $s$, which completes the proof.
\end{proof}

The above result provides a better rate in $n^{-2s}$ (instead of $n^{-s}$) but slightly worse exponent of the $\log(n)$ term. In the following, we will only exploit  the  result of \Cref{dp-mixed-sobolev-sparse}.

\begin{remark}
  Based on \cite{hansen2012best}, analogous results can be obtained for $L^p$-weighted widths with $H^{s}_{\mix}$ replaced by the Lizorkin-Triebel space $S^s_{p,2} F$ for $1 \leq p < \infty$.
\end{remark}

For each $\nu \in D$, we introduce a sequence of spaces $U_{\nu,n_\nu}$ with dimension $n_\nu$ (e.g. trigonometric polynomials or wavelets) such that for all $u\in H^{s}(\Xc_{\nu}),$ the error  $ E(u,U_{\nu,n_\nu})_{L^2_{\nu}}$ satisfies \eqref{approx-spaces-Unu}, which implies
 \[
\int_{\Xc_{\nu^c}} E\bigl(f^\nu(x_{\nu^c}),U_{\nu,n_\nu}\bigr)_{X_{\nu}}^2 \,d x_{\nu^c}  \le M^2 n_\nu^{-2s} \Vert f \Vert_{H^{s}_{\mix}}^2.
 \] 
Then we let $U_n = U_{1,n_1} \otimes \hdots \otimes U_{d,n_d}$. 
Now, we can state an approximation result for the approximation of functions in mixed Sobolev spaces using tree tensor networks. 

\begin{theorem}\label{approx-Ws2mix}
Let $f \in H^{s}_{\mix}(\Xc)$. Let $0<\varepsilon<1$.
We denote by $N(f,\varepsilon,d)$ the complexity $N(T,r,U_n)$ sufficient to achieve a relative error $\varepsilon$ for the approximation of $f$ in the format $\Tc^T_r(U_n).$ 
There exists a constant {$C_d$}, {which may depend exponentially on $d$}, such that 
\begin{enumerate}[{\rm(i)}]
\item if $T$ is a trivial tree with depth one,  
\[
N(f,\varepsilon,d) \le C_d   \varepsilon^{-d/(2s)} \log(\varepsilon^{-1})^{d(d-2)},
\]
\item 
and if $T$ is a binary tree,  
\[
N(f,\varepsilon,d) \le  C_d \varepsilon^{-3/(2s)} \log(\varepsilon^{-1})^{3(d-2)}.
\]
\end{enumerate}
\end{theorem}

\begin{proof}
From \cite[Lemma 1]{Schneider201456}, we know that for some function $c(\varepsilon,d)$ such that  $c(\varepsilon,d) \to 1$ as $\varepsilon\to 0$ and $c(\varepsilon,d) \to \infty$ super-exponentially with $d_\alpha$,
 the condition
 \begin{equation}\label{cond-ralpha-mix}
 r_\alpha \ge c(\varepsilon,d)  (C\sqrt{d+\#T-1})^{1/(2s)} s^{-d+2}  \varepsilon^{-1/(2s)} \log(\varepsilon^{-1})^{d-2} %
\end{equation}
implies 
$$
C r_\alpha^{-2s} \log(r_\alpha)^{2s(d-2)} \le \varepsilon / \sqrt{d+\#T-1}.
$$
Then using  \Cref{prop:best-approx-rU} and \Cref{dp-mixed-sobolev-sparse}, we have that 
if $n_{\nu} \ge \varepsilon^{-1/s}(M \sqrt{d+\#T-1})^{1/s}$ for all $1\le \nu \le d$, and $r_\alpha$ satisfies \eqref{cond-ralpha-mix} 
 for each node $\alpha \in T$,  then 
$$e^T_{r,U_n}(f)_{X} \le \varepsilon \Vert f \Vert_{{H^{s}_{\mix}}}.$$
The minimal values of ranks and dimensions $n_\nu$ such that the above conditions hold are such that 
$r_\alpha := r_\alpha(\varepsilon) \sim \varepsilon^{-1/(2s)} \log(\varepsilon^{-1})^{d-2}$, $\alpha \in T\setminus \{ D\}$, and $n_\nu := n_\nu(\varepsilon) \sim \varepsilon^{-1/s}$, $1\le \nu\le d$, with constants depending on $d$, $M$ and $s$. Then recalling that $N(f,r,U_{n})=\sum_{\nu=1}^d r_\nu n_\nu + \sum_{\alpha \in  \Ic(T)  } r_\alpha \prod_{\beta\in S(\alpha)} r_\beta$, we have 
\begin{align*}
N(f,\varepsilon,d) \lesssim \; &d \varepsilon^{-3/(2s)}  \log(\varepsilon^{-1})^{d-2} + \varepsilon^{-a/(2s)} \log(\varepsilon^{-1})^{a(d-2)}   \\
&+ 
(\#T-d-1) \varepsilon^{-(a+1)/(2s)} \log(\varepsilon^{-1})^{(a+1)(d-2)},\end{align*}
where $a = \max_{\alpha \in \Ic(T)} \#\alpha$ is the arity of the tree $T$.
\end{proof}

\begin{remark}
 From the above result, we can make the following observations.
 \begin{enumerate}[{\rm(i)}]
 \item {For a trivial tree (Tucker format), we have a complexity $\varepsilon^{-d/(2s)}$, up to a logarithmic factor, which compared to the result of \Cref{approx-Ws2} for $H^{s}$-regularity represents a deterioration by a factor two in the rate.} In other words, the   
extra regularity of functions in 
 $H^{s}_{\mix}$ compared to those of $H^{s}$ is not exploited by {shallow tensor networks. }
\item For binary trees, we observe a significant gain, going from a complexity in $\varepsilon^{-d/s}$ for {$H^{s}$} to a complexity in $\varepsilon^{-3/(2s)}\log(\varepsilon^{-1})^{3d-6}$ for 
{$H^{s}_{\mix}$}. The result is similar to the one obtained in \cite{Schneider201456} using results on bilinear approximation from \cite{Temlyakov:1989}. 
\item We note, however, that deep tensor networks (associated with binary trees) do not achieve the optimal rate in $\varepsilon^{-1/s}$ (up to logarithmic factors) obtained from linear widths of mixed Sobolev balls, and reached by hyperbolic cross approximation \cite{temlyakov1986approximations,dung2018hyperbolic}.  %
\end{enumerate}
\end{remark}

\begin{remark}The optimal rate in $\varepsilon^{-1/s}$ (up  to logarithmic terms) could be obtained by tree tensor networks by further exploiting sparsity in the tensors, and by using a measure of complexity $N$ counting the number of nonzero entries. In particular, this optimal rate can be achieved with a trivial tree and a tensor $C^D$  having a sparsity pattern based on hyperbolic crosses. We refer the reader to \cite{Ali2020partI,Ali2020partII,ali2021approximation} for the analysis of approximation classes of tensor networks with sparsity.
\end{remark}
\section{Approximation of compositional functions}\label{sec:composition}

We have seen in \Cref{sec:tree-tensor-networks} that tree tensor networks are a particular class of compositional functions, where the functions that are composed are vector-valued multilinear functions. 
  In this section, we consider the approximation with tree tensor networks of a particular class of compositional functions (also considered in \cite{mhaskar2016deep})  where the functions that are composed are real-valued functions with Sobolev regularity. 
  In this section, we consider a set $\Xc = \Xc_1\times \hdots \times \Xc_d$, with $\Xc_\nu = I^\nu$ a bounded and closed interval, equipped with the uniform measure. The results can be easily extended to the case of more general measures.
     
\subsection{A class $\Fc^T_s$ of compositional functions}
We let $T$ be a given dimension partition tree over $D=\{1,\hdots,d\}.$  
We consider the model class $\Fc^T$ of compositional functions $f : \Xc \to \Rbb$ of the form 
$$
f(x) = f_D((g_\alpha(x_\alpha))_{\alpha \in S(D)}) 
$$
where $g_\alpha : \Xc_\alpha \to I^\alpha \subset \Rbb$ and 
 $f_D$ is a multivariate function with values in $\Rbb=:I^D$, where $g_\alpha(x_\alpha) = x_\alpha$ for $\alpha \in \Lc(T)$, and for $\alpha \in \Ic(T)$, 
$$
g_{\alpha}(x_\alpha) = f_\alpha((g_\beta(x_\beta))_{\beta\in S(\alpha)})
$$
where $g_\beta : \Xc_\beta \to I^\beta \subset \Rbb$ and $f_\alpha$ is a multivariate function.
The function $f$ is completely determined by the set of multivariate functions  
\[  f_\alpha : \bigtimes_{\beta \in S(\alpha)} I^\beta\to I^\alpha ,\quad  \alpha \in \Ic(T) .  \]
 \begin{example}
For the dimension tree $T$ of Figure \ref{fig:example_tree}, the function $f $ admits the representation
$$
f(x) = f_{\{1,2,3,4,5\}}(f_{\{1,2,3\}}(x_1,f_{\{2,3\}}(x_2,x_3)),f_{\{4,5\}}(x_4,x_5)).
$$
Note that for $\alpha \in \Ic(T) \setminus \{D\}$, we can take $I^{\alpha} = [-\Vert f_\alpha \Vert_{L^\infty},\Vert f_\alpha \Vert_{L^\infty}]$. 

\end{example}
With $\Ic_\ell(T) = \{ \alpha \in \Ic(T)\colon \level(\alpha) = \ell \}$ the set of interior nodes with level $\ell$ and $\mathcal{V}_\ell = \Ic_\ell(T) \cup \{ \alpha \in \Lc(T) \colon \level(\alpha)\leq \ell \}$, we define the compositions of a given function $G_{\ell} \colon \bigtimes_{\beta \in \Vc_{\ell+1}} I^\beta \to I^D$ with all $f_\alpha$, $\alpha \in \Ic_{\ell+1}(T)$, by
\begin{multline}\label{eq:levelwisecomposition}
 G_{\ell} \circ_{\ell} (f_\alpha)_{\alpha \in \Ic_{\ell+1}(T)}
   = \left( (x_\alpha)_{\alpha \in \Vc_{\ell+1}} \mapsto  G_{\ell}\bigl((x_\alpha)_{\alpha \in \Vc_{\ell+1}} \bigr) \big|_{x_\beta = f_{\beta} ( (x_{\gamma})_{\gamma \in S(\beta)} ) \text{ for $\beta \in \Ic_{\ell+1}(T)$}} \right) .
\end{multline}
Starting with $\mathcal{C}_0 ({\{ f_\alpha\}_{\alpha\in \Ic(T)}}) = f_D$, we now recursively define the compositions of all $f_\alpha$ up to a given level $\ell$ in $T$ by
\[
 \mathcal{C}_\ell ( {\{ f_\alpha\}_{\alpha\in \Ic(T)}} ) 
   =   \\
     \mathcal{C}_{\ell-1} \bigl( {\{ f_\alpha\}_{\alpha\in \Ic(T)}} \bigr) \circ_{\ell-1} (f_\alpha)_{\alpha \in \Ic_{\ell}(T)}.
\]
We denote by $\mathcal{C}$ the map which associates to the entire set of functions $\{f_\alpha\}_{\alpha\in \Ic(T)}$ the compositional function $f \in \Fc^T$,
  \[
    f= \mathcal{C}(\{f_\alpha\}_{\alpha\in \Ic(T)}) =  \mathcal{C}_{\depth(T)-1}(\{f_\alpha\}_{\alpha\in \Ic(T)}) .
  \] 
We now restrict the class to functions with parameters $f_\alpha$ having Sobolev regularity  $W^{s,\infty}$ {with $s \in \Nbb$} by introducing 
 \[
  \Fc_s^T = \{f= \mathcal{C}(\{f_\alpha\}_{\alpha\in \Ic(T)}) : f_\alpha \in W^{s,\infty},\alpha\in \Ic(T)\}.
  \]
Next, we introduce a subset of $\Fc_s^T$ where the norms of parameters $f_\alpha$ are controlled. 
For a given $B=(B_1,\hdots,B_s) \ge (1,\hdots,1)$, we {define} 
\begin{align*}
\Fc^T_{s,B} = \bigl\{f \in \Fc^T :\; &\Vert f_D \Vert_{W^{s,\infty}} \le 1, \\
&\Vert D^{k_\alpha} f_\alpha \Vert_{L^\infty} \le B_{\vert k_\alpha \vert} \text{ for  $1 \le \vert k_\alpha \vert \le s$  and $\alpha \in \Ic(T) \setminus \{D\}$} \bigr\} ,
\end{align*}
where $k_\alpha$ is a multi-index in $\Nbb^{\#S(\alpha)}$. 
For any
$f\in \Fc^T_{s,B}$, we have $\Vert D^{k_\alpha} f_\alpha \Vert_{L^\infty} \le B_1$ for $\vert {k_\alpha} \vert \le 1$. Using the chain rule, we can prove that $
\Vert   f \Vert_{W^{s,\infty}} \le p_{T,s,B},
$ with $p_{T,s,B}$ depending on the tree $T$, on $s$ and on $B$.

\subsection{Approximation of  functions in $\Fc^T_{s}$ using tree tensor networks}

\subsubsection{An approach based on linear widths}\label{sec:compositional-linear-widths}

Let $T$ be a dimension tree over $D=\{1,\hdots,d\}$ and consider a compositional function $f\in \Fc^T_{s,B}$.

\begin{lemma}\label{lem:composition_hi}
Let $f \in \Fc^T_{s,B}$. For $\alpha \in T \setminus \{D\}$, $$f(x) = F_\alpha(g_\alpha(x_\alpha),x_{\alpha^c}),$$ with 
$F_\alpha : I^\alpha \times \Xc_{\alpha^c} \to \Rbb $ such that for any fixed $x_{\alpha^c}$,   
 $F_\alpha(\cdot,x_{\alpha^c}) \in S_{\level{(\alpha)}}^{s,\infty,B}$ with   
$$
S_\ell^{s,\infty,B} =  \{h = h_1\circ \hdots \circ h_\ell : \Vert h_1\Vert_{W^{s,\infty}}\le 1 , \Vert D^j h_i\Vert_{L^\infty} \le B_j, 0\le j\le s,2\le i \le \ell\} .$$
\end{lemma}

\begin{proof}
Let $\alpha \in T\setminus \{D\}$. Let $P(\alpha) \in T$ be the parent node of $\alpha$ in $T$ (such that $\alpha \in S(P(\alpha))$), let $B(\alpha) \subset T$ be the set of brothers of $\alpha$ (such that $B(P(\alpha)) \cup \{\alpha\} = S(\alpha)$), and let $A_\alpha$ be the ancestors of $\alpha$,  which is of cardinality
$
\#A_\alpha = %
\level(\alpha).
$  
We let $F_\alpha : I_\alpha \times \Xc_{\alpha^c} \to \Rbb $ be the function such that 
$$f(x) = F_\alpha(y_\alpha , x_{\alpha^c}) \quad  \text{with} \quad y_\alpha = g_\alpha(x_\alpha).$$ 
Letting $\ell = \level(\alpha)$, and letting $A_\alpha = \{\beta_1 , \hdots , \beta_{\ell} \}$ be the {ancestors}  of $\alpha$ ordered by increasing level (that is, {$D = \beta_1  \supset \hdots \supset \beta_{\ell} = P(\alpha) $}), the function $F_\alpha$ admits the representation
$$ F_\alpha(t, x_{\alpha^c}) = f_{\beta_1}( f_{\beta_2}(\hdots f_{\beta_\ell}(t, (y_\beta)_{\beta \in B(\alpha)})\hdots ),(y_{\beta})_{\beta \in B(\beta_2)}),$$
with  $y_\beta = g_\beta(x_\beta)$. 
Therefore, for a fixed $x_{\alpha^c}$, the function $F_\alpha(\cdot , x_{\alpha^c}) : I_\alpha \to \Rbb $ can be written as $F_\alpha(t , x_{\alpha^c}) = h_1 \circ \hdots \circ h_\ell (t)$, where $h_{i}(\cdot) = f_{\beta_i}(\cdot , (y_\beta)_{\beta \in B(\beta_{i+1})}) : I_{\beta_{i+1}} \to I_{\beta_{i}}$
satisfies
$
h_i \in W^{s,\infty}.
$ 
We have $h_1(\cdot) = f_D(\cdot,(y_\beta)_{\beta \in B(\beta_{2})}) $, so that $\Vert h_1 \Vert_{W^{s,\infty}} \le \Vert f_D \Vert_{W^{s,\infty}}   \le 1. $ And for $2 \le i\le \ell$ and $1\le j \le s$, $\Vert D^j h_i \Vert_{L^\infty} = \Vert D^{(j,0)} f_{\beta_i} \Vert_{L^\infty} $, with $(j,0) \in \Nbb^{\#S(\beta_i)}$, so that  $\Vert D^j h_i \Vert_{L^\infty} \le B_j$.
\end{proof}

We consider the sets of partial evaluations
\[
  K_\alpha(f) = \bigl\{f(\cdot,x_{\alpha^c}) : x_{\alpha^c} \in \Xc_{\alpha^c} \bigr\} =  \{F_\alpha ( g_\alpha(\cdot) , x_{\alpha^c}) : x_{\alpha^c} \in \Xc_{\alpha^c}\},
\] 
for which we have the following width estimate.

\begin{lemma}\label{lem:bound_width1}
For $f \in \Fc^T_{s,B}$ and any $\alpha \in T \setminus \{D\}$, 
$$
d_k(K_\alpha(f))_{L^\infty(\Xc_{\alpha})} \le d_{k}\bigl(S_{\level{(\alpha)}}^{s,\infty,B}\bigr)_{L^\infty(I^\alpha)}.
$$
\end{lemma}

\begin{proof}[Proof of \Cref{lem:bound_width1}]
Recall that $K_\alpha(f) = \{F_\alpha ( g_\alpha(\cdot) , x_{\alpha^c}) : x_{\alpha^c} \in \Xc_{\alpha^c}\}$, with $g_\alpha : \Xc_\alpha \to I_\alpha$. Therefore
\begin{align*}
d_k(K_\alpha(f))_{L^\infty(\Xc_{\alpha})} &= \inf_{\dim(V_\alpha)=k} \sup_{x_{\alpha^c}} \inf_{v \in V_\alpha}  \Vert F_\alpha ( g_\alpha(\cdot) , x_{\alpha^c})  - v(\cdot) \Vert_{L^\infty(\Xc_{\alpha})}
\\
&\le \inf_{\dim(W)=k} \sup_{x_{\alpha^c}} \inf_{w \in W}   \Vert F_\alpha ( g_\alpha(\cdot) , x_{\alpha^c})  - w(g_\alpha(\cdot)) \Vert_{L^\infty(\Xc_{\alpha})},
\end{align*}
where the  inequality has been obtained by restricting the minimization over $k$-dimensional subspaces $V_\alpha = \{w(g_\alpha(\cdot)) : w \in W\}$, with $W$ a $k$-dimensional subspace of functions defined on $I_\alpha$.
Then, introducing 
$K_1(F_\alpha) = \{F_\alpha(\cdot,x_{\alpha^c}) : x_{\alpha^c} \in \Xc_{{\alpha^c}}\} \subset L^\infty(I_\alpha)$,
we have 
\begin{align*}
d_k(K_\alpha(f))_{L^\infty(\Xc_{\alpha})} &\le   \inf_{\dim(W)=k} \sup_{h \in K_1(F_\alpha)} \inf_{w \in W}   \Vert h ( g_\alpha(\cdot)) - w ( g_\alpha(\cdot))\Vert_{L^\infty(\Xc_{\alpha})}
\\
&\le \inf_{\dim(W)=k} \sup_{h \in K_1(F_\alpha)} \inf_{w \in W}   \Vert h   - w  \Vert_{L^\infty(I_\alpha)}
\\
&= d_k(K_1(F_\alpha))_{L^\infty(I_\alpha)}.
\end{align*}
The result now follows from the fact that  $K_1(F_\alpha) \subset S_{\level{(\alpha)}}^{s,\infty,B}$.
\end{proof}

\begin{lemma}\label{lem:composition_sobolev}
For $h= h_1 \circ \hdots \circ h_\ell  \in  S_\ell^{s,\infty,B}$, we have
$$
\Vert h \Vert_{W^{s,\infty}} \le C(B,s,\ell),
$$
where $C(B,1,\ell) =   B_1^{\ell-1}$, $C(B,2,\ell) =  \ell B_1^{2\ell-2} B_2,$ and more generally, 
$$
C(B,s,\ell) =  (C \ell)^{s-1} B_1^{s(\ell-1)} B_\star^s%
$$
with $C = B_\star=1$ for $s=1$, and $B_\star = \max_{1\le j\le s } B_j$ and $C\ge 1$ for $s\ge 2$. 
\end{lemma}

\begin{proof}
We first note that $\Vert h \Vert_{L^\infty} \le \Vert h_1\Vert_{L^\infty} \le 1$.
Let $h_{>i}(t) = h_{i+1} \circ \hdots  \circ h_\ell(t) $. 
 Then
$h'(t) = \prod_{i=1}^\ell g_{1,i}(t),$ with $g_{1,i}(t) = h_i'\circ h_{>i}(t)$.  
 Therefore $\Vert h' \Vert_{L^\infty} \le B_1^{\ell-1}$. Then we have $h''(t) = \sum_{i=1}^\ell g_{1,i}'(t) \prod_{j\neq i} g_{1,j}(t)$, with $g_{1,i}'(t) = h_i''(h_{>i}(t)) \prod_{j>i} g_{1,j}(t) := g_{2,i}(t).$ Then, we have  
$\Vert h''\Vert_{L^\infty} \le \sum_{i=1}^\ell B_2 B_1^{2\ell-i-1} \le \ell B_2 B_1^{2(\ell-1)}.$ 
\end{proof}

From \Cref{lem:bound_width1,lem:composition_sobolev} and \Cref{kolmogorov-width-sobolev-balls}, we directly obtain the following result.
\begin{lemma}\label{lem:k-width-composition}
For $f \in \Fc^T_{s,B}$ and any node $\alpha \in T \setminus \{D\}$ with level $\ell_\alpha$,
\[
d_n(K_\alpha(f))_{L^\infty(\Xc_{\alpha})}  \le  R C(B,s,\ell_\alpha) n^{-s}, \quad n \in \Nbb,
\]
with a constant $R$ not depending on $\ell_\alpha$, $B$ and $d$. 
\end{lemma}

For each $\nu \in D$, we introduce a sequence of spaces $U_{\nu,n_\nu}$ with dimension $n_\nu$ (e.g. splines or wavelets) such that for all $u\in H^{s}(\Xc_\nu),$
 \begin{align}\label{composition-Unu}
 E(u , U_{\nu,n_\nu})_{X_{\nu}} \le M  n_\nu^{-s}\Vert u \Vert_{H^{s}},
\end{align}
with $M\ge R$, with $R$ the constant from \Cref{lem:k-width-composition}.  
\begin{lemma}\label{lem:approx-leaves} Let $f\in \Fc^T_{s,B}$ and $\nu \in D$. For all $u\in K_{\{\nu\}}(f)$, with $\ell_\nu = \level({\{\nu\}})$, we have
\[
E(u , U_{\nu,n_\nu})_{X_{\nu}} \le M C(B,s,\ell_\nu) n_\nu^{-s}
\]
with $M$ a constant not depending on $d$.
\end{lemma}
\begin{proof}
The set of partial evaluations $K_{\{\nu\}}(f)$ satisfies $K_{\{\nu\}}(f) \subset H^{s}(\Xc_\nu)$, so that \eqref{composition-Unu} holds for all $u\in K_{\{\nu\}}(f)$. Also 
$K_{\{\nu\}}(f) \subset S^{s,\infty,B}_{\ell_\nu}$. Therefore, from \Cref{lem:composition_sobolev}, we have $\Vert u \Vert_{H^{s}} \le C(B,s,\ell_\nu)$ for all $u \in  K_{\{\nu\}}(f)$, which completes the proof.
\end{proof}

\begin{proposition}\label{prop:L2approx}
Let $f\in \Fc^T_{s,B}$. For an admissible rank $r\in \Nbb^{\#T}$ and $U_r = U_{1,r_1} \otimes \hdots \otimes U_{d,r_d}$, we have
\[
 e^T_{r,U_r}(f)_{X}^2 \le  \sum_{\alpha \in T \setminus {D} } \bigl( M  C(B,s,\level(\alpha)) \bigr)^2 \,r_{\alpha}^{-2s}\,.
\]
\end{proposition}

\begin{proof}
This follows from \Cref{prop:best-approx-rU}, the bound $\delta^\alpha_n ( f)  \le d_n(K_\alpha(f))_{L^\infty(\Xc_{\alpha})}$ which holds since $\operatorname{meas}(\Xc) = 1$,  \Cref{lem:k-width-composition}, and \Cref{lem:approx-leaves}.
\end{proof}

\subsubsection{A constructive approach using uniform approximations}

For each $\alpha \in \Ic(T)$,  let $(\mathcal{Q}^\alpha_{N})_{N \in \mathbb{N}^{\#S(\alpha)}}$ be a family of linear operators mapping $C(\bigtimes_{\beta \in S(\alpha)} I^\beta)$ to a finite-dimensional tensor subspace spanned by product basis functions,
\[
  \mathcal{Q}^\alpha_{N}\colon  C(\bigtimes_{\beta \in S(\alpha)} I^\beta) \to \bigotimes_{\beta \in S(\alpha)}  U_{\beta,N_\beta}
  \]
  with 
  $$
    {U_{\beta,N_\beta}}  =  \operatorname{span}  \{ {\varphi^{\beta}_{N_{\beta}, i}\colon i=1,\ldots,N_{\beta}}\}, \quad {\beta \in S(\alpha).}
  $$
 We assume these operators to have the properties 
 \begin{equation}\label{eq:nonexpQ}
  \| \mathcal{Q}^\alpha_N g \|_{L^\infty} \leq  \| g  \|_{L^\infty}
 \end{equation} 
   and for all $s \in ( 0, s^*]$ with $s^* \in (0,\infty]$,  $\min N :=  \min_{\beta \in S(\alpha)} N_\beta$,
 \begin{equation}\label{eq:approxQ}
     \| g - \mathcal{Q}^\alpha_N g \|_{L^\infty} \leq Q_{\#S(\alpha)}  {\bigl( \min N\bigr)^{-s}}  \| g \|_{W^{s,\infty}}.
      \end{equation}
Here $Q_{\#S(\alpha)} >0$ is independent of $N$ and $g$, but may depend on $\#S(\alpha)$, where $Q_{\#S(\alpha)} \leq Q_a$ for a $Q_a>0$ whenever $\#S(\alpha) \leq a$. The operators $\mathcal{Q}^\alpha_N$ are thus required to be non-expansive and provide approximations in $L^\infty$-norm converging at optimal rate up to some maximum order.

\begin{example}
The operators $\mathcal{Q}^\alpha_N$ can be chosen as piecewise constant interpolation on a uniform partition into $N_\beta$ subintervals in the coordinate $\beta$, in which case \eqref{eq:approxQ} holds for $s \in (0,1]$; or piecewise linear interpolation with $s \in (0, 2]$.
\end{example}

In general, $\mathcal{Q}^\alpha_N g$ is of the form
\[
    \mathcal{Q}^\alpha_{N} g = \sum_{i_{1}, \ldots, i_{a}}  c^\alpha_{i_{1}, \ldots, i_{a}}(g)\, \varphi^{\beta_1}_{N_{\beta_1}, i_1} \otimes\cdots\otimes \varphi^{\beta_a}_{N_{\beta_a}, i_a} , \quad S(\alpha) = \{ \beta_1, \ldots, \beta_a\} 
\]  
with coefficients $c^\alpha_{i_{1}, \ldots, i_{a}}(g) \in \mathbb{R}$. 

For $f \in \mathcal{F}^T_{s,B}$ with $f= \mathcal{C}(\{f_\alpha\}_{\alpha\in \Ic(T)})$, for given tuples of positive integers $N_\alpha \in \Nbb^{\#S(\alpha)}$, $\alpha \in T$, we define $\tilde f_\ell$ for $\ell = 0,\ldots,\depth(T)-1$ recursively as follows:
\[
      \tilde f_0 = \mathcal{Q}^D_{N_D}  f_D,
\]
and for $\ell > 0$, with $T_\ell = \{ \alpha_1, \ldots, \alpha_{d_\ell} \}$,
\[
    \tilde f_\ell = \biggl( \bigotimes_{\alpha \in \Vc_\ell} \mathcal{Q}^{\alpha}_{N_{\alpha}}  \biggr) \bigl(  \tilde f_{\ell - 1} \circ_{\ell-1} ( f_{\alpha} )_{\alpha \in \Ic_{\ell}} \bigr).
 \]
We set $\tilde f = \tilde f_{L-1}$ with $L = {\depth(T)}$.
\begin{lemma}\label{lem:Linftyapprox}
For $f\in \Fc^T_{s,B}$,
\begin{equation}\label{eq:comperrestLinfty0}
\Vert f - \tilde f \Vert_{L^\infty} \leq   \sum_{\alpha \in \Ic(T) } Q_{\#S(\alpha)}  C\bigl(B,s,\level(\alpha)\bigr) \, \bigl( \min {N_{\alpha}} \bigr)^{-s} \,.
\end{equation}
\end{lemma}
\begin{proof}
Let $f_\ell = \mathcal{C}_\ell(( f_\alpha )_{\alpha \in \Ic_\ell(T)})$, that is, $f_\ell$ are the compositions of the functions $f_\alpha$ up to level $\ell$ without approximations, so that $f = f_{L-1}$. We set 
\[
  \mathcal{Q}_{\ell} = \bigotimes_{\alpha \in \Ic_\ell} \mathcal{Q}^{\alpha}_{N_\alpha}\, \otimes\, \bigotimes_{\alpha \in \Vc_\ell \setminus \Ic_\ell} \id_\alpha
  \]
and note that $\| \mathcal{Q}_{\ell}\|  \leq 1$ by \eqref{eq:nonexpQ} and that by the triangle inequality, for any $h$,
\[
 \|  h - \mathcal{Q}_{\ell} h  \|_\infty \leq  \sum_{\alpha \in \Ic_\ell}  \|   h - ( \mathcal{Q}^{\alpha}_{N_\alpha} \otimes \id_{\alpha^c} ) h  \|_\infty \,.
\]
Since $f = f_{L-2} \circ_{L-2}  (f_{\alpha} )_{\alpha \in \Ic_{L-1}(T)}$, combining the above and \eqref{eq:approxQ} with Lemma \ref{lem:composition_sobolev} we obtain
  \[  
  \begin{aligned}
   \| f - \tilde f\|_{L^\infty}  
     &\leq \| f - \mathcal{Q}_{L-1} f \|_{L^\infty}   \\
     &\qquad + \|  \mathcal{Q}_{L-1} ( f_{L-2} \circ_{L-2}  (f_{\alpha} )_{\alpha \in \Ic_{L-1}(T)} ) - \mathcal{Q}_{L-1} ( \tilde f_{L-2} \circ_{L-2}  (f_{\alpha} )_{\alpha \in \Ic_{L-1}(T)} ) \|_{L^\infty} \\
      & \leq
        \sum_{\alpha \in \Ic_{L-1}}  \| (I -  ( \mathcal{Q}^{\alpha}_{N_\alpha} \otimes \id_{\alpha^c} )) f \|_{L^\infty} 
           +   \| \mathcal{Q}_{L-1}\|   \| f_{L-2} - \tilde f_{L-2} \|_{L^\infty}  \\
           & \leq  \sum_{\alpha \in \Ic_{L-1}} C(B,s,{\level(\alpha)})   Q_{\#S(\alpha)} {\bigl( \min N_{\alpha} \bigr)}^{-s}  + \| f_{L-2} - \tilde f_{L-2} \|_{L^\infty} .
   \end{aligned}
  \]
  Applying the same argument to $\ell < L-1$ starting with $ f_{L-2} - \tilde f_{L-2}$, we recursively obtain
  \[
       \| f - \tilde f\|_{L^\infty}  \leq   {\sum_{\ell = 0}^{L-1}   \sum_{\alpha \in \Ic_\ell} Q_{\#S(\alpha)} C(B,s,\level(\alpha))  {\bigl( \min N_{\alpha} \bigr)}^{-s} \,, }
  \]
  which completes the proof.
\end{proof}
From the above, we deduce a result on the approximation with tree tensor networks in $L^\infty$-norm.
\begin{proposition}\label{prop:Linftyapprox}
Let $f\in \Fc^T_{s,B}$ {with $s>0$, and let \eqref{eq:nonexpQ} and \eqref{eq:approxQ} hold for this $s$}. For an admissible rank $r\in \Nbb^{\#T}$ and $U_r = U_{1,r_1} \otimes \hdots \otimes U_{d,r_d}$, we have
\begin{equation}\label{eq:comperrestLinfty}
e_{r,U_r}^T(f)_{L^\infty} \leq   \sum_{\alpha \in T\setminus\{D\} } Q_a C (B,s,\level(\alpha) ) \, r_{\alpha}^{-s} ,
\end{equation}
with $a$ the arity of $T$.
\end{proposition}
\begin{proof}
We let $N_\alpha = (r_\beta)_{\beta \in S(\alpha)}$ for all $\alpha\in \Ic(T)$ and $\tilde f$ be the corresponding approximation defined above, which is such that $\tilde f \in U_r$ and $\rank_\alpha(\tilde f) \le r_\alpha$ for each 
$\alpha \in T \setminus \{D\}$. Therefore, $e_{r,U_r}^T(f)_{L^\infty} \leq \Vert f - \tilde f \Vert_{L^\infty}$ and the result follows from \Cref{lem:Linftyapprox} and the fact that for each $\alpha \in \Ic(T)$, $Q_{\#S(\alpha)} \le Q_a$, $\bigl( \min N_{\alpha}\bigr)^{-s} = \max_{\beta \in S(\alpha)} r_\beta^{-s} \le \sum_{\beta \in S(\alpha)} r_{\beta}^{-s}$ and for each $\beta \in S(\alpha)$, $C(B,s,\level(\alpha)) \le C(B,s,\level(\beta))$. 
\end{proof}

\begin{remark}
Similar results can still be obtained when the assumptions \eqref{eq:nonexpQ} and \eqref{eq:approxQ} are relaxed. One example is for each $\alpha \in T$ to choose $\mathcal{Q}^\alpha_N$ as the Lagrangian interpolation operator on $\bigtimes_{\beta \in S(\alpha)} I^\beta$ corresponding to interpolation in Chebyshev points $\{ x^\beta_1,\ldots, x^\beta_{ {N_\beta} } \}$ on each $I^\beta$; that is, if $S(\alpha)$ contains only interior nodes in the tree, $\mathcal{Q}^\alpha_{N}$ acts on $g \in C(\bigtimes_{\beta \in S(\alpha)} I^\beta)$ as
\[
    \mathcal{Q}^\alpha_{N} g = \sum_{i_{1}, \ldots, i_{a}} g( x^{\beta_1}_{i_{1}}, \ldots, x^{\beta_a}_{i_a} ) \, {\varphi^{\beta_1}_{N_{\beta_1}, i_1} \otimes\cdots\otimes \varphi^{\beta_a}_{N_{\beta_a}, i_a} }, \quad S(\alpha) = \{ \beta_1, \ldots, \beta_a\}
\]  
where $\varphi^\beta_{N, i}$,  $i=1,\ldots,N$ are the Lagrange basis polynomials for the Chebyshev points on $I^\beta$. 
For the Lebesgue constant $\Lambda_N$, we have $\Lambda_N \leq {\prod_{\beta \in S(\alpha)} (\frac{2}{\pi} \log ( N_\beta + 1) + 1)}$.
Recall that $\| \mathcal{Q}^\alpha_N g \|_{L^\infty} \leq \Lambda_N \| g \|_{L^\infty}$ and by Lebesgue's lemma,
\[
  \|  g - \mathcal{Q}^\alpha_N g \|_{L^\infty} \leq ( 1 + \Lambda_N ) \min_{p \in \Pi_N} \| g - p \|_{L^\infty} 
    \lesssim \bigl( \min {N} \bigr)^{-s} \big(\prod_{\beta\in S(\alpha)}( 1 + \log N_\beta) \big) \| g \|_{W^{s,\infty}},
\]
with $\Pi_N = \bigotimes_{\beta\in S(\alpha)} U_{\beta,N_\beta}$.
Thus \eqref{eq:nonexpQ} and \eqref{eq:approxQ} both hold only up to an additional logarithmic factor. This leads to additional factors in $\log(r_\alpha)^a$ on the right in \eqref{eq:comperrestLinfty}. 
\end{remark}

For  a given $r = (r_\alpha)_{\alpha \in T}$, we let $N_\alpha = (r_\beta)_{\beta \in S(\alpha)}$ for each $\alpha \in \Ic(T)$ and $\tilde f$  the corresponding approximation. For $\alpha \in \Ic(T) $, the component tensor of the tree network representation of $\tilde f$ at node {$\alpha$} is explicitly given for 
$\alpha \neq D$ by
\[
  A^{{\alpha}}_{j, i_{1}, \ldots, i_{{\#S(\alpha)}}}   :=  c^{{\alpha}}_{i_{1}, \ldots, i_{\#S(\alpha)}} \bigl(  { \varphi^{\alpha}_{{r_\alpha}, j}   \circ f_\alpha} \bigr), \quad j \in\{ 1,\ldots, {r_\alpha}\},\;  {i_\beta \in  \{1,\ldots, {r_\beta}\}},  {\beta \in S(\alpha),} 
\]
or for $\alpha=D$ by
\[
  A^{D}_{ i_{1}, \ldots, i_{a}}   :=  c^{D}_{i_{1}, \ldots, i_{\#S(D)}} \bigl(  f_D \bigr), \quad  i_\beta \in  \{1,\ldots, r_\beta\},  \beta \in S(D).
\]
\begin{example}
  Let $D = \{1, 2,3, 4\}$ and let $T$ be the corresponding balanced tree with arity $a = 2$. Then we have the explicit tensor representation
  \[
  \begin{aligned}
    \tilde f &= {\sum_{i_{12}=1}^{r_{12}} \sum_{i_{34} = 1}^{r_{34}} } {\sum_{i_1=1}^{r_1} \sum_{ i_2 = 1}^{r_2}}   {\sum_{i_3=1}^{r_3} \sum_{ i_4 = 1}^{r_4}}    c^D_{i_{12}, i_{34}} (f_D) \, c^{\{1,2\}} _{{ i_1, i_2}} ( \varphi^{\{1,2\}}_{{r_{12}} , i_{12}} \circ f_{\{1,2\}} ) \,  c^{\{3,4\}} _{{ i_3, i_4}} ( \varphi^{\{3,4\}}_{{r_{34}} , i_{34}} \circ f_{\{3,4\}} )   \\
   & \qquad\qquad \qquad \qquad\qquad \qquad  \times    \varphi^{\{1\}}_{{r_1} , i_1} \otimes \varphi^{\{2\}}_{{r_2} , i_2} \otimes \varphi^{\{3\}}_{{r_3} i_3} \otimes \varphi^{\{4\}}_{{r_4}, i_4}   \\
     & =  \sum_{\substack{ i_{12}, i_{34} \\  i_1, i_2 ,  i_3, i_4}  }  A^D_{i_{12}, i_{34}}   A^{\{1,2\}} _{i_{12}, i_1, i_2}  A^{\{3,4\}} _{i_{34}, i_3, i_4} \, 
     {\varphi^{\{1\}}_{r_1, i_1} \otimes \varphi^{\{2\}}_{r_2, , i_2} \otimes \varphi^{\{3\}}_{r_3 , i_3} \otimes \varphi^{\{4\}}_{r_{4} , i_4}  }.
  \end{aligned}
  \]
  With the particular choice of $\mathcal{Q}^\alpha_N$ as piecewise constant approximation, with each $\varphi^{\beta}_{{r_\beta}, j} $ the characteristic function of a subinterval of $I^\beta$, the entries of the tensors $A^\alpha$ of order three have a simple interpretation: their nonzero entries correspond exactly to parallelepipeds in the chosen three-dimensional product grid that intersect the graph of the bivariate function $f_\alpha$; in other words, these entries mark a ``voxel approximation'' of the graph of $f_\alpha$.
\end{example}

\subsubsection{Approximation complexity estimates}

We are now ready to state the main result on the approximation of compositional functions from $\Fc^T_{s,B}$ by tree tensor networks. 
 \begin{theorem}\label{th:N-compositional}
 Let $f \in \Fc^T_{s,B}$ with $s \in \Nbb$. For $\epsilon>0$, we denote by $N(f,\varepsilon,d)$ the complexity $N(T,r,U_n)$ sufficient to achieve an error $\varepsilon$ for the approximation of $f$ in the format $\Tc^T_r(U_n),$ with error measured in $L^2$ for arbitrary $s$ or in $L^\infty$ for {$s \leq 2$}. 
Let $a =  \max_{\alpha \in \Ic(T)} \#S(\alpha)$ be the arity of the tree and $L=\depth(T)$. 
Then we have the following estimates:
\begin{enumerate}[{\rm(i)}]
\item For a trivial tree $T$ with arity $a=d$ and depth $L=1$, 
$$
N(f,\varepsilon,d)\le {C_d}\,  \varepsilon^{-d/s}.
$$
with a constant $C_d$ depending super-exponentially on $d$ but not depending on $\epsilon$. 
\item For a tree with arity $a$ independent of $d$,  
\begin{align}
N(f,\varepsilon,d) \le {C_d}\, {L^{a+1}}  \varepsilon^{-(a+1)/s} B_1^{(a+1)L} B_\star^{a+1}\label{N-compo-arity-a-bounded}
\end{align}
with a constant $C_d$ depending polynomially on $d$ but not depending on $\varepsilon$. 
\end{enumerate}
\end{theorem}

\begin{proof}
{From \Cref{prop:L2approx} and \Cref{prop:Linftyapprox}, we have that 
$$e_{r,U_n}^T(f)_{L^p} \leq   \sum_{\alpha \in T\setminus\{D\} } M C(B,s,\level(\alpha)) \, r_{\alpha}^{-s} ,$$
for $p=2$ and arbitrary $s$ with a constant $M$ independent of $d$, and for $p=\infty$ and $s=1$ or $2$ and a constant $M$ depending  
 on the arity $a$.
}
If the ranks $r_\alpha$ are such that 
  \begin{align}
  r_\alpha \ge \varepsilon^{-1/s}(\#T-1)^{1/{s}} M^{1/s}  C(B,s,\ell_\alpha)^{1/s}  ,\label{cond-r}\end{align}
 then $
e^T_{r,{U_r}}(f)_{{L^p}}  \le \varepsilon
$.
From  \Cref{lem:composition_sobolev}, we know that $C(B,s,\ell) \le (C \ell)^{s-1} B_1^{s(\ell-1)} B_\star^s .$ 
Therefore, from condition \eqref{cond-r}, we deduce the sufficient condition on $r_\alpha$ to achieve an error $\varepsilon$ is 
$$
r_\alpha \ge  M^{1/s}\varepsilon^{-1/s} (C\ell_\alpha)^{1-1/s} B_1^{\ell_\alpha-1}B_\star (\#T-1)^{1/{s}}.
$$
Letting $r_\alpha := r_\alpha(\varepsilon) \sim M^{1/s}\varepsilon^{-1/s} (C\ell_\alpha)^{1-1/s} B_1^{\ell_\alpha-1}B_\star (\#T-1)^{1/(2s)}$ be the minimal ranks satisfying the above condition, we have $N(f,\varepsilon,d) \le N(T,r,U_r)$, which yields 
\begin{multline*}
N(f,\varepsilon,d)  \lesssim M^{a/s}\varepsilon^{-a/s} (C)^{a-a/s} B_\star^a (\#T-1)^{a/{s}}   \\
 \qquad + \sum_{\alpha \in \Ic(T) \setminus \{D\}}  M^{(a+1)/s}\varepsilon^{-(a+1)/s} (C\ell_\alpha)^{1-1/s}(C(\ell_\alpha+1))^{a-a/s} \\
 \qquad\qquad
 \times B_1^{(a+1)\ell_\alpha-1} B_\star^{a+1} (\#T-1)^{(a+1)/{s}} 
\\ + \sum_{\nu=1}^d M^{2/s}\varepsilon^{-2/s} (C\ell_\nu)^{2-2/s} B_1^{2\ell_\nu-2}B_\star^2 (\#T-1)^{{2}/s}.
\end{multline*}
 Noting that $\#T\le 2d-1$, $\ell_\alpha \le L-1$ for $\alpha \in \Ic(T)$ and $\ell_\alpha \le L$ for $\alpha\in \Lc(T)$, and letting $a =  \max_{\alpha \in \Ic(T)} \#S(\alpha)$ be the arity of the tree and $L=\depth(T)$, we obtain 
  \begin{align*}
N(f,\varepsilon,d) \le& M^{a/s}\varepsilon^{-a/s} (C)^{a(1-1/s)} B_\star^a (\#T-1)^{a/{s}} \nonumber  \\
& 
+ (\#\Ic(T)-1)  M^{(a+1)/s}\varepsilon^{-(a+1)/s} (CL)^{(a+1)(1-1/s)} B_1^{(a+1)(L-1)-1} B_\star^{a+1} (2d)^{(a+1)/{s}} 
\nonumber\\
& +d M^{2/s}\varepsilon^{-2/s} (CL)^{2(1-1/s)} B_1^{2L-2}B_\star^2 (2d)^{{2}/s}. %
\end{align*}
In particular, for a trivial tree $T$ with arity $d$ and depth $1$, 
\[
N(f,\varepsilon,d)\le M^{d/s}\varepsilon^{-d/s} (C)^{d(1-1/s)} B_\star^d d^{d/{s}} +d M^{2/s}\varepsilon^{-2/s} (CL)^{2-2/s} B_\star^2 (2d)^{{2}/s},
\]
whereas for a tree with arity $a$ independent of $d$,  
\begin{equation*}
N(f,\varepsilon,d) \le \beta d^{1+(a+1)/{s}}  \varepsilon^{-(a+1)/s} L^{a+1} B_1^{(a+1)L} B_\star^{a+1}%
\end{equation*}
with a constant $\beta$ independent of $\varepsilon$ and $d$.
 \end{proof}

\begin{remark}
The following observations can be made:
\begin{enumerate}[{\rm(i)}]
\item 
As expected, we observe that for a trivial tree, a shallow tensor network (Tucker format) does not exploit more than the Sobolev regularity $H^{s}$ of the function and suffers from the curse of dimensionality.
\item 
In the case where the tree has arity {$a$ independent of $d$}, we observe in \eqref{N-compo-arity-a-bounded}
that the complexity is exponential in the depth $L$ through the term 
$B_1^{(a+1)L} $, which depends on the bound $B_1$ on the first derivatives of functions $f_\alpha$. 
\item 
An important observation is that if {$a$ is independent of $d$} and $B_1 = 1$ (i.e. functions $f_\alpha$ are $1$-Lipschitz), then there is no more an exponential dependence on $L$ and the complexity depends polynomially on $d$ and $\varepsilon^{-1}$. That means that tree tensor networks do not present the curse of dimensionality for functions in $ \Fc^T_{s,B}$. 
 \item 
 When $B_1>1$ and $a$ independent of $d$, tree tensor networks may or may not suffer from the curse of dimensionality for functions in $ \Fc^T_{s,B}$, depending on the dependence of $L$ in $d$.
 \item For binary trees with $a=2$, 
 \begin{align*}
 N(f,\varepsilon,d)  \le  {C_d L^3} \varepsilon^{-3/s} B_1^{3L} B_\star^{3}  ,
\end{align*}
where $L\le \lceil \log_2(d) \rceil$ for a balanced binary  tree, and $L = d-1$ for a linear binary tree.
For a linear tree, we observe a complexity exponential in $d$ of the form $B_1^{3d}$. However, for a balanced tree, the dependence in $d$ is only polynomial on $d$ of the form $B_1^{3\log_2(d)}$. This means that the approximation complexity may depend exponentially on $d$ for any tree with depth depending polynomially on $d$, in particular for linear trees, but remains polynomial in $d$ for balanced trees, or more generally for any tree with a depth $L$ depending logarithmically on $d$.
\end{enumerate}
\end{remark}

\bibliographystyle{plain}

 \end{document}